\documentclass{article}
\usepackage[british]{babel}
\usepackage[utf8]{inputenc}
\usepackage{xcolor}
\usepackage{caption,subcaption,longtable}
\usepackage{graphics, setspace}
\usepackage{mathtools}
\usepackage{array,multirow}
\usepackage{tikz,float}
\usepackage{amsmath,amsthm,amssymb}
\usepackage{enumerate,enumitem,listings}
\usepackage[T1]{fontenc}
\usepackage{stackengine,scalerel}
\usepackage{hyphenat}
\usepackage{slantsc}
\usepackage[hang, flushmargin]{footmisc}
\usepackage{hyperref}
\usepackage{footnotebackref}
\allowdisplaybreaks

\makeatletter
\newtheorem*{rep@theorem}{\rep@title}
\newcommand{\newreptheorem}[2]{%
\newenvironment{rep#1}[1]{%
 \def\rep@title{#2 \ref{##1}}%
 \begin{rep@theorem}}%
 {\end{rep@theorem}}}
\makeatother
\makeatletter
\newtheorem*{rep@lemma}{\rep@title}
\newcommand{\newreplemma}[2]{%
\newenvironment{rep#1}[1]{%
 \def\rep@title{#2 \ref{##1}}%
 \begin{rep@lemma}}%
 {\end{rep@lemma}}}
\makeatother

\newtheorem{theoremalph}{Theorem}

\newtheorem{coralph}[theoremalph]{Corollary}
\newtheorem{theorem}{Theorem}[section]
\newtheorem{lemma}[theorem]{Lemma}
\newreptheorem{theorem}{Theorem}
\newreplemma{lemma}{Lemma}
\theoremstyle{definition}
\newtheorem{definition}[theorem]{Definition}
\newtheorem{remark}[theorem]{Remark}

\setlist[enumerate]{align=left}
\allowdisplaybreaks
\title{Random groups do not have Property (T) at densities below $1\slash 4$}
\author{Calum J. Ashcroft}
\date{\vspace{-2em}}

\begin{document}
\maketitle
	\begin{abstract}
We prove that random groups in the Gromov density model at density $d<1\slash 4$ do not have Property (T), answering a conjecture of Przytycki. We also prove similar results in the $k$-angular model of random groups.
	\end{abstract}


\section{Introduction}

The \emph{density model} of random groups was introduced by Gromov in \cite{gromovasymptotic} to study a `typical' group. Fix $n\geq 2, l\geq 1$, and $0<d<1$. Let $F_{n} $ be the free group of rank $n$. A \emph{random group on $n$ generators at density $d$ and length $l$} is defined by killing a uniformly random chosen set of $(2n-1)^{ld}$ cyclically reduced words of length $l$ in $F_{n}$. We write $G\sim G(n,l,d)$ if a random group has the distribution of this procedure, and \emph{with overwhelming probability} (\emph{w.o.p.}) if a property holds with probability tending to $1$ as $l$ tends to infinity. 

The study of Property (T) in random quotients was proposed in Gromov's manuscript on hyperbolic groups \cite[$\S$4.5C]{Gromov_hyperbolic}. A famed theorem of \.{Z}uk states that w.o.p a random group at density $d>1\slash 3$ has Property (T) \cite{Zuk} (see also \cite{kotowski,ashcroft2021property}). \.{Z}uk's theorem was recently extended to random quotients of hyperbolic groups in \cite{ashcroft2022propertyTrandomquotients}, answering a question of Gromov--Ollivier. 

Random groups at density less than $1\slash 6$ w.o.p. act freely and cocompactly on a CAT(0) cube complex \cite{Ollivier-Wise} (they are hyperbolic w.o.p. \cite{gromovasymptotic,olliviersharp} and hence linear w.o.p. \cite{Haglund-Wise,Agol13}). Montee proved that random groups w.o.p. act non-trivially and cocompactly on a CAT(0) cube complex at density less than $3\slash 14$ \cite{montee2021random}, extending the previous density bounds for such an action of Ollivier--Wise ($d<1\slash 5$, \cite{Ollivier-Wise}) and Mackay--Przytycki ($d<5\slash 24$, \cite{mackay_przytycki2015balanced}).

A gap exists between the current known density bound for Property (T) and for an unbounded action on a CAT(0) cube complex. It has been conjectured by Przytycki\footnote{This conjecture was personally communicated to the author by Przytycki and has been strongly alluded to in, for example, \cite[p. 398]{mackay_przytycki2015balanced}.} that random groups at densities below $1\slash 4$ do not have Property (T). We prove this conjecture.
\begin{theoremalph}\label{thmalph: no property t in random groups}
Let $G$ be a random group at density less than $1\slash 4$. Then, with overwhelming probability, $G$ acts with unbounded orbits on a finite dimensional CAT(0) cube complex, and hence does not have Property (T).
\end{theoremalph}
However, the question of what happens between densities $1\slash 4$ and $1\slash 3$ remains. Indeed, random groups at density $d<1\slash 6$ w.o.p. have the Haagerup property \cite[Corollary $9.2$]{Ollivier-Wise}, a strong negation of Property (T). There are therefore two questions to answer. Firstly, what is the optimum density bound for random groups to have the Haagerup property? And secondly, what is the optimum density bound for random groups to have Property (T)? More concretely, suppose that $0<d_{H}\leq d_{(T)}<1$ are optimal constants such that a random group at density $d<d_{H}$ w.o.p. has the Haagerup property and a random group at density $d>d_{(T)}$ w.o.p. has Property (T). Does $d_{H}=d_{(T)}$? This would meant that the Haagerup property and Property (T) would be `generically opposite', as discussed in \cite[$\S IV.c$]{Ollivier_Jan_Invitation}. Beyond Property (T), there are results on further fixed point properties in the density model, such as $F(L^{p})$ \cite{DRUTU_Mackay}, or actions on uniformly negatively curved Banach spaces \cite{Oppenheim21} (see also \cite{de2021banach} for similar results in the \emph{triangular model}). Questions then remain for the optimum density bounds for these fixed point properties.

  Our approach to the problem is different to the previous methods used to obstruct Property (T) in random groups. The results of Ollivier--Wise, Mackay--Przytycki, and Montee were proved by building separating subspaces of the Cayley complex of a random group, and then employing the cubulation techniques of Sageev \cite{Sageev-95}. These subspaces were constructed inductively by joining midpoints of edges to form `hypergraphs', using the isoperimetric inequality for random groups (see \cite{ollivier2007some,odrzygozdzsquare}) to prove that these hypergraphs are $2$-sided trees. As $d$ tends to $1\slash 4$ these methods present increasing combinatorial difficulties. We instead approach the problem similarly to Osajda's work on group cubization \cite{osajda2018group}, using covers of Cayley graphs to obstruct Property (T) in certain extensions of groups. 
  
 In Definition \ref{def: aspherically separated}, we introduce the notion of a group being \emph{aspherically relator separated}. It is well known that a group acting with unbounded orbits on a finite dimensional CAT(0) cube complex does not have Property (T) \cite[Theorem B]{Niblo-Reeves97}. The main technical result of this paper is Theorem \ref{thmalph: well separated implies no (T)}, the proof of which occupies Section \ref{sec: Using covers of graphs to obstruct (T)}. 
\begin{theoremalph}\label{thmalph: well separated implies no (T)}
Let $G$ be an aspherically relator separated group. Then $G$ acts with unbounded orbits on a finite dimensional CAT(0) cube complex, and hence does not have Property (T).
\end{theoremalph}

We may therefore deduce Theorem \ref{thmalph: no property t in random groups} from Theorem \ref{thmalph: well separated implies no (T)} and the following lemma, the proof of which forms Section \ref{sec: Random groups are aspherically relator separated}.

\begin{replemma}{lem: random groups are well separated}
Let $G$ be a random group at density less than $1\slash 4.$ Then, with overwhelming probability, $G$ is aspherically relator separated. 
\end{replemma}
We study a further model of random groups that was introduced in \cite{ashcroftrandom}. Fix $k\geq 2$ and $0<d<1.$ A \emph{random group in the $k$-angular model at density $d$}
is defined by choosing $G\sim G(n,k,d)$ and letting $n$ tend to infinity. This was preceded by the \emph{triangular model} of \.{Z}uk \cite{Zuk} and the \emph{square model} of Odrzyg{\'o}{\'z}d{\'z} \cite{odrzygozdzsquare}. The \emph{positive $k$-angular model}, $G^{+}(n,k,d)$, is obtained similarly, by choosing $n^{kd}$ positive words (words containing no inverse letters) of length $k$ in $F_{n}$. For any fixed $m$ and $k$, a random group $G\sim G^{+}(n,k,d)$ admits a homomorphism onto a finite index subgroup of a random group $H\sim G(m,nk,d)$ \cite[$\S 3.3$]{kotowski}. Since Property (T) is preserved by passing to quotients and finite index extensions, we obtain the following corollary (writing $w.o.p.(n)$ to mean that a property holds with probability tending to $1$ as $n$ tends to infinity).
\begin{coralph}
Let $G$ be a random group in the positive $k$-angular model at density $d<1\slash 4$. Then, w.o.p.(n), $G$ does not have Property (T).
\end{coralph}

It was recently shown that random groups in the $k$-angular model at density $d>(k+(-k\mod 3))\slash 3k$  have Property (T) w.o.p.(n) \cite{ashcroft2021property}, extending the known results of: $d>1\slash 3$ for the triangular model \cite{Zuk}; $d>1\slash 3$ for the hexagonal model \cite{odrzygozdz2019bent}; and in fact $d>1\slash 3$ for any $k$-angular model where $k$ is divisible by $3$ \cite{montee(T)}. Random groups do not have Property (T) w.o.p.(n) at densities: $d<1\slash 3$ in the triangular model \cite{Zuk}; $d<3\slash 8$ in the square model \cite{odrzygozdz2019bent}; and $d<1\slash 3$ in the hexagonal model \cite{odrzygozdz2019bent}. Random groups in the square model are in fact w.o.p.(n) virtually special (in the sense of Haglund--Wise) at density $d<1\slash 3$ \cite{duong}. We prove the following result for the $k$-angular model. 
\begin{theoremalph}\label{thmalph: lack of (T) in $k$-angular model}
Let $k\geq 2$ and $d<\lfloor k\slash 2\rfloor\slash2k.$ Let $G$ be a random group in the $k$-angular model at density $d$. Then, w.o.p.(n), $G$ acts with unbounded orbits on a finite dimensional CAT(0) cube complex, and hence does not have Property (T).
\end{theoremalph}

We are dealing with asymptotics, hence frequently arrive at situations where $m$ is some parameter tending to infinity that is required to be an integer. If $m$ is not integer, then we will implicitly replace it by $\lfloor m\rfloor$. This will not affect the validity of any of our arguments.

\section*{Acknowledgements}
 I would like to thank John Mackay for comments on an earlier draft and Piotr Przytycki for comments on an earlier draft and communicating his conjecture. I am grateful to Emmanuel Breuillard for interesting conversations regarding Property (T) in random groups. As always, I would like to thank Henry Wilton for his invaluable advice, as well as comments on an earlier draft of this work.
\section{Using covers of graphs to obstruct (T)}\label{sec: Using covers of graphs to obstruct (T)}
We first discuss the construction of a specific covering of a given Cayley graph. Our construction is similar to Osajda's group cubization \cite{osajda2018group}, which employs Wise's double cover \cite[$\S9.a$]{Wise-MSQT}. Throughout, given a set $S$, $F(S)$ will be the free group on the set $S$. Consider subsets $T, R\subseteq F(S) $ with $T\subseteq \langle R\rangle$. Let $G_{T}=\langle S\vert T\rangle $ and $K_{R}=\langle S\vert R\rangle,$ so that there is an epimorphism $\xi:G_{T}\twoheadrightarrow K_{R}$ obtained by mapping each $s\in S$ in $G_{T}$ to the generator $s$ in $K_{R}$. Let $\Gamma_{T}$ be the Cayley graph of $G_{T}$ and $\Sigma_{R}$ the Cayley graph of $K_{R},$ so that $\pi_{1}(\Gamma_{T})=\langle\langle T\rangle\rangle $ and $\pi_{1}(\Sigma_{R})=\langle\langle R\rangle\rangle $.  We have the inclusion $\langle\langle T\rangle\rangle\hookrightarrow \langle\langle R\rangle\rangle,$ and hence an induced map $\iota:H_{1}(\langle\langle T\rangle\rangle)\rightarrow H_{1}(\langle\langle R\rangle\rangle)$. All homology groups in this text are taken with coefficients in $\mathbb{Z}\slash 2\mathbb{Z}$, and so homology groups are $\mathbb{F}_{2}$ vector spaces.

\begin{definition}
The \emph{double T-cover} of $\Sigma_{R}$ is the covering graph $\pi:\hat{\Delta}_{T,R}\looparrowright \Sigma_{R}$ corresponding to the kernel of the map $$P_{T,R}:\pi_{1}(\Sigma_{R})=\langle\langle R\rangle\rangle\twoheadrightarrow H_{1}(\langle\langle R\rangle\rangle)\twoheadrightarrow Q(T,R)=H_{1}(\langle\langle R\rangle\rangle)\slash \iota(H_{1}(\langle\langle T\rangle\rangle)).$$
Since $Q(T,R)=\bigoplus_{I}\mathbb{Z}\slash 2\mathbb{Z}$ for some index set $I$, $\hat{\Delta}_{T,R}$ is a normal cover of the Cayley graph $\Sigma_{R}$. It follows immediately by the construction and Sabidussi's theorem \cite{SabidussiGert1958OaCo} that $\hat{\Delta}_{T,R}$ is the Cayley graph of the group $\hat{H}_{T,R}=F(S)\slash \ker(P_{T,R})$. We also call $\hat{H}_{T,R}$ the \emph{double $T$-cover} of $K_{R}$ (note that $G_{T}\twoheadrightarrow \hat{H}_{T,R}\twoheadrightarrow K_{R}$). 
\end{definition}
\begin{remark}
The double cover is obtained by taking $T=\emptyset$. In particular, given a group $K_{R}=\langle S\vert R\rangle$, Osajda's group cubization $\tilde{H}$ is the double $T$-cover of $K_{R}$ for $T=\emptyset$. 
\end{remark}

\begin{remark}
    There are several different groups occurring in the above constructions. Particular care should be taken that we have \emph{surjections} $$G_{T}=\langle S\vert T\rangle \twoheadrightarrow \hat{H}_{T,R}=\langle S\vert \ker(P_{T,R})\rangle \twoheadrightarrow K_{R}=\langle S\vert R\rangle, $$
    whereas the relators give us \emph{inclusions} of subgroups of the free group $F(S)$, $$\langle \langle T\rangle\rangle \leq \ker (P_{T,R})\leq \langle \langle R\rangle \rangle.$$
    
    In an attempt to reduce confusion where possible, we make the following conventions. 
    \begin{itemize}
        \item Elements of $G_{T}$ will always be denoted by $g$, elements of $\hat{H}_{T,R}$ will always be denoted by $h$, and elements of $K_{R}$ will always be denoted by $k$ or $\kappa$.
    
   \item Elements of $T$ will always be written $t$ or $\tau$, whereas elements of $R$ will always be written $r$ or $\rho$.
    
   \item Finally, paths in $\Gamma_{T}$ will be denoted by $\gamma$, paths in $\hat{\Delta}_{T,R}$ will be denoted by $\delta$, and paths in $\Sigma_{R}$ will be denoted by $\sigma.$
      \end{itemize}
\end{remark}

We now move to the central definition of this text, designed to give us control over the quotient group $Q(T,R)$, and hence allows us to also obtain control over the covering graph $\hat{\Delta}_{T,R}$. A group presentation $L=\langle X\vert Y\rangle$ is \emph{aspherical} if the corresponding presentation complex $\mathcal{P}(X,Y)$ is topologically aspherical, or equivalently, if the second homology with integer coefficients of the universal cover of $\mathcal{P}(X,Y)$, $\tilde{\mathcal{P}}(X,Y)$, is trivial.
\begin{definition}\label{def: aspherically separated}
Let $L=\langle S\vert  t_{1},\hdots , t_{N}\rangle$ be a finite presentation such that each $t_{i}$ is reduced without cancellation, and we have fixed a partition $t_{i}=r_{i}r_{i}'$ for each word $t_{i}$. Let $\mathcal{H}(L)=\langle S\vert r_{1},r_{1}',\hdots , r_{N},r_{N}'\rangle .$
We say that $L$ is \emph{aspherically relator separated} if the presentations $L$ and $\mathcal{H}(L)$ are aspherical, and $\mathcal{H}(L)$ is infinite. 
\end{definition}
Since $L$ surjects $\mathcal{H}(L)$, the above also implies that $L$ is infinite. 

\begin{remark}
    In the sequel, unless otherwise specified, we will always fix the partition $t_{i}=r_{i}r_{i}'$ with $\vert r_{i}\vert=\vert r_{i}'\vert$ if $\vert t_{i}\vert$ is even, and $\vert r_{i}\vert=\vert r_{i}\vert-1$ otherwise. 
\end{remark}

We begin by noting the following consequences of a group presentation being aspherical.
\begin{lemma}\label{lem: consequences of asphericity}\cite{Chiswell_asphericity}
    Let $L=\langle X\vert Y\rangle$ be an aspherical group presentation.
    \begin{enumerate}[label=$\roman*)$]
        \item No $y\in Y$ is a proper power, and if $y\in Y$ then no other member of $Y$ is a cyclic conjugate of $y$ or $y^{-1}$ (including the trivial conjugates $y$, $y^{-1}$).
        \item If $y,y',\gamma,\gamma'\in Y$ with $y'\neq y^{-1}$ and $yy'=\gamma\gamma'$, then $y=\gamma$ and $y'=\gamma'$.
    \end{enumerate}
\end{lemma}
\begin{proof}
Statement $i)$ follows by \cite[Proposition $1.3$]{Chiswell_asphericity}. 

For Statement $ii)$, we use the statement from \cite[proof of Lemma 1.8]{Chiswell_asphericity}, that Condition (I.1) of \cite[$\S III.10.2$]{LyndonCGT} is valid for $L$. Put simply, this implies that, since $yy'\gamma'^{-1}\gamma^{-1}=1$ with $y\neq y^{-1}$, either: $y=\gamma$ and $y'=\gamma'$; or $y=\gamma'$ and $y'=\gamma$. 

Suppose that $y=\gamma',\;y'=\gamma$. Then in the free group $F(X)$ we have $yy'=y'y$. In particular, $y$ and $y'$ commute, and so are powers of a common element $v$. Since $y,y'$ cannot be proper powers by statement $i)$, we have $y=y'=v,$ so that $y=y'=\gamma=\gamma'$. In particular, in either case we have $y=\gamma$ and $y'=\gamma'$.\end{proof}

\begin{theorem}\cite[Proposition 1.2]{Chiswell_asphericity}
Let $L=\langle X\vert  Y\rangle$ be an aspherical group presentation and let $\overline{Y}=\langle\langle Y\rangle\rangle^{ab}$ be the relation module. Then $\overline{Y}$ decomposes as a free $L$-module into a direct sum of cyclic submodules $M_{y}$, $y\in Y$, where each $M_{y}$ is generated by $\overline{y}=y[\langle\langle Y\rangle\rangle,\langle\langle Y\rangle\rangle]$.
\end{theorem}

In particular, $H_{1}(\langle\langle Y\rangle\rangle)$ is easily described.

\begin{remark}\label{rmk: first homology for aspherical presentations}
    Let $L=\langle X\vert  Y\rangle$ be an aspherical group presentation. Then $$H_{1}(\langle\langle Y\rangle\rangle) =\bigoplus_{\ell\in L}\bigoplus_{y\in Y}\mathbb{Z}\slash 2\mathbb{Z}(\ell\overline{y}).$$
\end{remark}

\begin{remark}
    For the remainder of this section, we will assume that we have fixed an aspherically relator separated group $G=\langle S\vert T\rangle$. We further assume that $T=\{t_{1},\hdots ,t_{N}\}$, where we have fixed for each $i$ a partition $t_{i}=r_{i}r_{i}'$, so that $\mathcal{H}(G)=\langle S\vert R\rangle$ for $R=\{r_{1},r_{1}',\hdots ,r_{N},r_{N}'\}.$ Using our earlier notation, we write $G=G_{T}$ and $\mathcal{H}(G)=K_{R}$. We construct the double $T$-covers: $\hat{\Delta}_{T,R}$ for $\Sigma_{R}$; and $\hat{H}_{T,R}$ for $K_{R}$. 
\end{remark}
We begin by analysing the group $Q(T,R)=H_{1}(\langle\langle R\rangle\rangle)\slash \iota(H_{1}(\langle\langle T\rangle\rangle))$. Recall that we have the surjection $\xi: G_{T}\twoheadrightarrow K_{R}$, and for $\rho\in R$ we defined $\overline{\rho}:=\rho [\langle \langle R\rangle\rangle,\langle \langle R\rangle\rangle]$. Let $\Psi: H_{1}(\langle \langle R\rangle\rangle) \twoheadrightarrow Q(T,R)$ be the obvious surjection. 

\begin{lemma}\label{lem: relations in Q(T,R)}
    For each basis element $k\overline{r}$ of $$H_{1}(\langle\langle R\rangle\rangle)=\bigoplus_{\kappa\in K_{R}}\bigoplus_{\rho\in R} \mathbb{Z}\slash2\mathbb{Z}\left(\kappa\overline{\rho}\right)$$ as an $\mathbb{F}_{2}$ vector space, there exists exactly one other basis element $k'\overline{r'}$ of $H_{1}(\langle\langle R\rangle\rangle)$ with $\Psi(k\overline{r})=_{Q(T,R)}\Psi(k'\overline{r'})$. This basis element takes the form $k\overline{r'}$, where $r'\neq r$ is the unique element of $R$ with $rr'$ (or $r'r$) lying in $T$.
\end{lemma}
Since $K_{R}$ is infinite, $\langle \langle R\rangle\rangle$ is an infinite index normal subgroup of $F(S).$ Therefore, $\langle \langle R\rangle\rangle$ is an infinite rank free group, and so $H_{1}(\Sigma_{R})=\langle\langle R\rangle\rangle^{ab}$ is infinite rank. Lemma \ref{lem: relations in Q(T,R)} implies that $Q(T,R)$ is of infinite rank, i.e. that it is infinite dimensional as an $\mathbb{F}_{2}$ vector space.
\begin{proof}
  Using Remark \ref{rmk: first homology for aspherical presentations}, we see that $\iota(H_{1}(\langle\langle T\rangle\rangle))$ is the subspace of $$H_{1}(\langle\langle R\rangle\rangle)=\bigoplus_{\kappa\in K_{R}}\bigoplus_{\rho\in R}\mathbb{Z}\slash2\mathbb{Z}\left(\kappa\overline{\rho}\right)$$ spanned by the set $\mathcal{V}=\left\{\left(\xi(g)\overline{\rho}+\xi(g)\overline{\rho'}\right):g\in G_{T},\rho,\rho'\in R, \rho\rho'\in T\right\}.$
  We have that $\Psi (k\overline{r})=_{Q(T,R)}\Psi (k'\overline{r'})$ if and only if $(k\overline{r}+k'\overline{r'})$ lies in the span of $\mathcal{V}$. By Lemma \ref{lem: consequences of asphericity} and Remark \ref{rmk: first homology for aspherical presentations}, the vectors in $\mathcal{V}$ are pairwise orthogonal, that is when adding two distinct elements of $\mathcal{V}$, no cancellation occurs.  Therefore,  $(k\overline{r}+k'\overline{r'})\in Span(\mathcal{V})$ if and only if $k=k'$ and either: $rr'\in T$ or $r'r\in T$; or $r=r'$ (in this latter case $k\overline{r}+k'\overline{r}'=_{Q(T,R)}\boldsymbol{\underline{0}})$. By Lemma \ref{lem: consequences of asphericity}, there is exactly one $r'\in R$ such that $rr'$ or $r'r$ is an element of $T$. Furthermore, by Lemma \ref{lem: consequences of asphericity}, no $t\in T$ is a proper power, so $r\neq r'$.
\end{proof}
We now find cutsets in $\hat{\Delta}_{T,R}.$ We will make great use of the following loops.
\begin{definition}\label{def: the loop sigma r}
Let $r$ be an element of $R$. We define $\sigma_{r}$ as the loop based at the identity in $\Sigma_{R}$ and read by $r$. 
\end{definition}

\begin{lemma}\label{lem: preimages are separating}
    Let $\pi:\hat{\Delta}_{T,R}\looparrowright \Sigma_{R}$ be the double $T$-cover, and for each $r\in R$ let $\sigma_{r}$ as the loop based at the identity in $\Sigma_{R}$ and read by $r$. For each word $t=rr'\in T$, and each $k\in K_{R}$, the set $\pi^{-1}(k\sigma_{r}\cup k\sigma_{r'})$ is a separating subset of $\hat{\Delta}_{T,R}$. 
\end{lemma}

Note that, for $t=rr'\in T$, the loop $\sigma_{r}\sigma_{r'}$ in $\Sigma_{R}$ lifts to a loop $\delta_{r,r'}=\tilde{\sigma}_{r}\tilde{\sigma}_{r'}$ in $\hat{\Delta}_{T,R}$; the content of the above lemma is that the orbit of this loop under the deck group forms a cutset. This is illustrated in Figure \ref{fig: cutsets}.

\begin{figure}[H]
  \centering 
  	\includegraphics[scale=0.7]{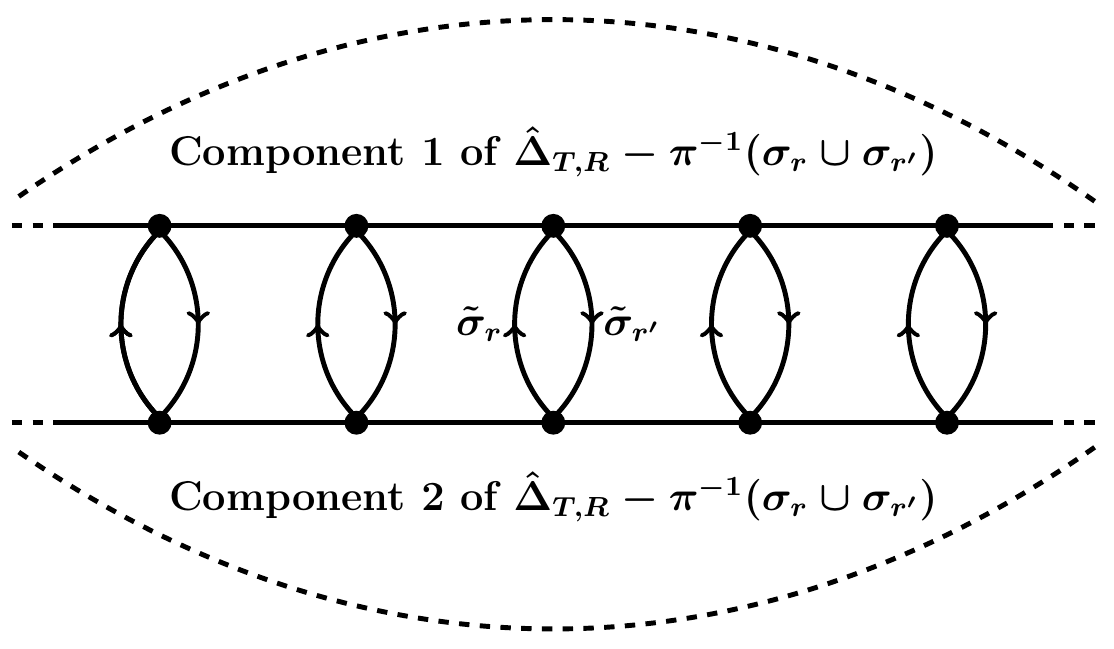}
    \caption{\centering The cutset $\pi^{-1}(\sigma_{r}\cup \sigma_{r'})$}\label{fig: cutsets}
\end{figure}

\begin{proof}

By Lemma \ref{lem: relations in Q(T,R)}, we see that the set $\{\Psi(\kappa\overline{\rho})\;:\;\kappa\in K_{R},t=\rho\rho'\in T\}$ forms a basis for the group $Q(T,R)$ when viewed as an $\mathbb{F}_{2}$ vector space. In particular, an element $\underline{\boldsymbol{x}}\in Q(T,R)$ may be written as $$\underline{\boldsymbol{x}}=\sum\limits_{\kappa\in K_{R},t=\rho\rho'\in T}\underline{\boldsymbol{x}}(\kappa,\rho)\Psi(\kappa\overline{\rho}).$$ Let $\langle \cdot,\cdot\rangle$ be the symmetric bilinear form on $Q(T,R)$, defined as follows. Given elements $\underline{\boldsymbol{x}},\;\underline{\boldsymbol{y}}\in Q(T,R),$
define $$\langle \underline{\boldsymbol{x}},\underline{\boldsymbol{y}}\rangle=\sum_{\kappa\in K_{R},t=\rho\rho'\in T}\underline{\boldsymbol{x}}(\kappa,\rho)\underline{\boldsymbol{y}}(\kappa,\rho).$$

Now, let $\delta$ be a simple path in $\hat{\Delta}_{T,R}$ containing exactly one subpath, $\delta_{1}$, that is the lift of exactly one of the loops $k\sigma_{r}$ or $k\sigma_{r'}$. Suppose that $\delta'$ is a simple path in $\hat{\Delta}_{T,R}-\pi^{-1}(k\sigma_{r}\cup k\sigma_{r'})$ connecting the endpoints of $\delta$. Then $\delta \delta'^{-1}$ is a loop in $\hat{\Delta}_{T,R}$, and so gives us an element $\phi$ of $\pi_{1}(\hat{\Delta}_{T,R})\leq \langle \langle R\rangle\rangle$. Since $\phi$ lies in $\pi_{1}(\hat{\Delta}_{T,R})=\ker (P_{T,R})$, we have that $P_{T,R}(\phi)=_{Q(T,R)}\underline{\boldsymbol{0}}$.

Let $\phi_{\delta}\in \langle \langle R\rangle\rangle=\pi_{1}(\Sigma_{R})$ be the element corresponding to the loop $\pi(\delta)$ lying in $\Sigma_{R}$, and define similarly $\phi_{\delta'}$. Note that $P_{T,R}(\phi)=P_{T,R}(\phi_{\delta})+P_{T,R}(\phi_{\delta'}).$ By our assumptions, and as the set $\{\Psi(\kappa \overline{\rho}):\kappa\in K_{R},t=\rho\rho'\in T\}$ forms a basis for the group $Q(T,R)$ when viewed as an $\mathbb{F}_{2}$ vector space, we have that $\langle P_{T,R}(\phi_{\sigma}), \Psi (k\overline{r})\rangle =1$, and $\langle P_{T,R}(\phi_{\delta'}), \Psi(k\overline{r})\rangle=0$, so that $\langle P_{T,R}(\phi), \Psi (k\overline{r})\rangle=1,$ a contradiction as $P_{T,R}(\phi)=_{Q(T,R)}\underline{\boldsymbol{0}}$.

\end{proof}

We now describe how to find finite-sided cutsets in $\hat{\Delta}_{T,R}$. The first case we deal with is when $\hat{\Delta}_{T,R}-\pi^{-1}(\sigma_{r}\cup\sigma_{r'})$ contains exactly one infinite component. If $\Sigma_{R}-(\sigma_{r}\cup\sigma_{r'})$ contains more than one infinite component, then we will use a slightly different approach, which requires less work.

\begin{definition}\label{def: Arr}
Let $t=rr'\in T$ and let $\sigma_{r}$ (respectively $\sigma_{r'}$) be the loop based at the identity in $\Sigma_{R}$ and read by $r$ (respectively $r'$). Suppose $\Sigma_{R}-(\sigma_{r}\cup \sigma_{r'})$ contains exactly one infinite component. Let $A_{r,r'}$ be the union of $\sigma_{r}\cup\sigma_{r'}$ together with the finite connected components of $\Sigma_{R}-(\sigma_{r}\cup \sigma_{r'})$.
\end{definition}
By Lemma \ref{lem: preimages are separating}, $\hat{\Delta}_{T,R}-\pi^{-1}(A_{r,r'})$ contains at least two components. It is easy to see that there exists a finite uniform bound $D$ to the diameter of $A_{r,r'}$ as $rr'$ ranges over the elements of $T$.

\begin{remark}
  Suppose that $\Sigma_{R}-(\sigma_{r}\cup \sigma_{r'})$ contains exactly one infinite component. Then $\Sigma_{R}-A_{r,r'}$ consists of exactly one component, which is infinite. Therefore, all components of $\hat{\Delta}_{T,R}-\pi^{-1}(A_{r,r'})$ are infinite. 
\end{remark}

\begin{definition}\label{def: Lrr}
Suppose that for $rr'\in T$, $\Sigma_{R}-(\sigma_{r}\cup \sigma_{r'})$ contains exactly one infinite component, and let $A_{r,r'}$ be as in Definition \ref{def: Arr}. Let $\mathcal{L}_{r,r'}$ be the finite set of loops in $\Sigma_{R}$ of the form $k\sigma_{\rho}$, where $k\in K_{R}$, $\rho\in R$, and $k\sigma_{\rho}\cap A_{r,r'}$ is non empty. 
\end{definition}

There is a uniform upper bound $L$ to the size of $\mathcal{L}_{r,r'}$ as $rr'$ ranges across $T$. 

\begin{remark}
  Suppose that $\Sigma_{R}-A_{r,r'}$ consists of exactly one component. Then given any $x\in \Sigma_{R}-A_{r,r'}$ and any component $C$ of $\hat{\Delta}_{T,R}-\pi^{-1}(A_{r,r'})$, there exists $\hat{x}\in C$ with $\pi (\hat{x})=x.$
\end{remark}
We now define a particular element of $Q(T,R).$

\begin{definition}\label{def: the element q(x,y)}
Let $rr'\in T$ and suppose that $\Sigma_{R}-(\sigma_{r}\cup \sigma_{r'})$ contains exactly one infinite component. Let $A_{r,r'}$ be given as in Definition \ref{def: Arr}. Given $\hat{x},\hat{y}$ in $\hat{\Delta}_{T,R}-\pi^{-1}(A_{r,r'})$ with $\pi(\hat{x})=\pi(\hat{y})$, let $\boldsymbol{\underline{q}}(\hat{x},\hat{y})$ be the element of $Q(T,R)$ defined as follows. Let $\delta$ be a path in $\hat{\Delta}_{T,R}$ between $\hat{x}$ and $\hat{y}$, so that $\pi(\delta)$ is a loop in $\Sigma_{R}-A_{r,r'}$, and therefore gives and element $\phi_{\delta}$ in $\langle \langle R\rangle\rangle.$ The element $\boldsymbol{\underline{q}}(x,y)$ is defined as $P_{T,R}(\phi_{\delta})$.
\end{definition}
Note that the above definition does not depend on the choice of the path $\delta$. This element encodes a lot of information about the components of $\hat{\Delta}_{T,R}-\pi^{-1}(A_{r,r'})$, as we now show.

\begin{lemma}\label{lem: determining components}
 Let $rr'\in T$ be such that $\Sigma_{R}-(\sigma_{r}\cup \sigma_{r'})$ contains exactly one infinite component. Let $A_{r,r'}$ be given as in Definition $\ref{def: Arr}$ and $\boldsymbol{\underline{q}}(x,y)$ be given as in Definition \ref{def: the element q(x,y)}. Suppose that $\hat{x},\hat{x}',\hat{y},\hat{y}'$ are such that:
\begin{itemize}
    \item $\hat{x},\hat{x}'$ lie in the same component of $\hat{\Delta}_{T,R}-\pi^{-1}(A_{r,r'})$, and
    \item we have $\boldsymbol{\underline{q}}(\hat{x},\hat{y})=_{Q(T,R)}\boldsymbol{\underline{q}}(\hat{x}',\hat{y}')$. 
\end{itemize} Then the points $\hat{y},\hat{y}'$ lie in the same component of $\hat{\Delta}_{T,R}-\pi^{-1}(A_{r,r'})$. 
\end{lemma}
\begin{proof}
Since $\hat{x}$ and $\hat{x}'$ lie in the same component of $\hat{\Delta}_{T,R}-\pi^{-1}(A_{r,r'})$, there exists a path $\delta(\hat{x}',\hat{x})$ in $\hat{\Delta}_{T,R}$ from $\hat{x}'$ to $\hat{x}$ not intersecting $\pi^{-1}(A_{r,r'})$. Write $\delta(\hat{x},\hat{x}')$ for the inverse of $\delta(\hat{x}',\hat{x})$

Let $\delta(\hat{x},\hat{y})$ be a path in $\hat{\Delta}_{T,R}$ from $\hat{x}$ to $\hat{y}$ with inverse $\delta(\hat{y},\hat{x})$, and define similarly  $\delta(\hat{x}',\hat{y}')$, $\delta(\hat{y}',\hat{x}')$.

Let $\delta'$ be the lift of $\pi(\delta(\hat{x}',\hat{x}))$ to $\hat{\Delta}_{T,R}$ starting at $\hat{y}'$. This gives a path in $\hat{\Delta}_{T,R}$ starting at $\hat{y}$, defined as $$\zeta=\delta(\hat{y},\hat{x})\delta(\hat{x},\hat{x}')\delta(\hat{x}'\hat{y}')\delta'.$$
In particular, we go from $\hat{y}$ to $\hat{x}$, then backwards along $\delta(\hat{x}',\hat{x})$ to $\hat{x'}$ then to $\hat{y}'$ and then along $\delta$. Note that $\pi(\zeta)$ is a loop in $\Sigma_{R}$ and gives an element $\phi_{\zeta}\in \langle \langle R\rangle \rangle$. We similarly obtain elements $\phi_{\delta(\hat{x},\hat{x}'})$ etc. 

By the definition of $\boldsymbol{\underline{q}}(\hat{x},\hat{y})$, and by the assumptions in the statement of the lemma, $$P_{T,R}(\phi_{\delta(\hat{x},\hat{y})})=\boldsymbol{\underline{q}}(\hat{x},\hat{y})=\boldsymbol{\underline{q}}(\hat{x}',\hat{y}')=P_{T,R}(\phi_{\delta(\hat{x}',\hat{y}')}).$$ Furthermore, $\pi(\delta')=\pi(\delta(\hat{x}',\hat{x})).$

Since we are taking homology with $\mathbb{Z}\slash 2\mathbb{Z}$ coefficients,
we immediately deduce that
\begin{align*}
    P_{T,R}(\phi_{\zeta})&=_{Q(T,R)}\boldsymbol{\underline{0}}
\end{align*}

Hence $\pi(\zeta)$ lies in $\ker(P_{T,R})$, and so $\zeta$ is a loop in $\hat{\Delta}_{T,R}$. Therefore $\delta'$ connects $\hat{y}$ and $\hat{y}'$ in $\hat{\Delta}_{T,R}$. Since $\delta(\hat{x},\hat{x'})$ does not intersect $\pi^{-1}(A_{r,r'})$, neither does $\delta'$ and so $\delta'$ is a path between $\hat{y}'$ and $\hat{y}$ not intersecting $\pi^{-1}(A_{r,r'}).$ In particular, $\hat{y}$ and $\hat{y}'$ lie in the same component of $\hat{\Delta}_{T,R}-\pi^{-1}(A_{r,r'})$.
\end{proof}

\begin{remark}

Suppose that $\Sigma_{R}-(\sigma_{r}\cup \sigma_{r'})$ contains exactly one infinite component, so that $\Sigma_{R}-A_{r,r'}$ is connected. Consider the components of $\hat{\Delta}_{T,R}-\pi^{-1}(A_{r,r'})$, of which there are at least two by Lemma \ref{lem: preimages are separating}. Let $\hat{x}$ be a vertex of $\hat{\Delta}_{T,R}-\pi^{-1}(A_{r,r'})$. If the loop $(\kappa\gamma_{\rho})$ is not an element of $\mathcal{L}_{r,r'}$, let $\sigma$ be a simple path in $\Sigma_{R}-A_{r,r'}$ from $\pi(\hat{x})$ to the start point of $(\kappa\gamma_{\rho})$, which exists as $\Sigma_{R}-A_{r,r'}$ is connected. Then the lift of the path $\sigma (\kappa\gamma_{\rho})\sigma^{-1}$ starting at $\hat{x}$ ends at an element $\hat{y}$ with $\boldsymbol{\underline{q}}(\hat{x},\hat{y})=\Psi (\kappa\gamma_{\rho})$. 
\end{remark}

In particular, by repeatedly applying the above construction, we may deduce the the following lemma.

\begin{lemma}
    Let $t=rr'\in T$, and let $\hat{x}$ be a point in $\hat{\Delta}_{T,R}-\pi^{-1}(A_{r,r'}).$ Let $\boldsymbol{\underline{q}}\in Q(T,R)$ with $$\langle \boldsymbol{\underline{q}},\Psi(\kappa\sigma_{\rho})\rangle=0$$
for all $\kappa\sigma_{\rho}\in L_{r,r'}$. Then there exists a point $\hat{y}$ in the same connected component of $\hat{\Delta}_{T,R}-\pi^{-1}(A_{r,r'})$ as $\hat{x}$, and with 
$$\boldsymbol{\underline{q}}(\hat{x},\hat{y})=_{Q(T,R)}\boldsymbol{\underline{q}}.$$
\end{lemma}

Therefore, given a point $\hat{x}$ in $\hat{\Delta}_{T,R}-\pi^{-1}(A_{r,r'})$, we see that the connected component of $\hat{\Delta}_{T,R}-\pi^{-1}(A_{r,r'})$ in which $\hat{x}$ lies is uniquely determined by $\langle \boldsymbol{\underline{q}}(\hat{x},\hat{y}),\Psi (k\gamma_{\rho})\rangle$, where $\hat{y}$ ranges overt the connected components of $\hat{\Delta}_{T,R}-\pi^{-1}(A_{r,r'})$ and $k\gamma_{\rho}$ ranges over $\mathcal{L}_{r,r'}$. In particular, by applying Lemma \ref{lem: determining components}, we may deduce the following. 

\begin{lemma}\label{lem: finitely many connected components}
   Suppose that $\Gamma_{R}-(\sigma_{r}\cup \sigma_{r'})$ contains exactly one infinite component. Let $A_{r,r'}$ be defined as in Definition \ref{def: Arr} and $\mathcal{L}_{r,r'}$ be defined as in Definition \ref{def: Lrr}. There are at most $2^{\vert \mathcal{L}_{r,r'}\vert}$ connected components of $\hat{\Delta}_{T,R}-\pi^{-1}(A_{r,r'})$.
\end{lemma}
Now, we begin to discuss the CAT(0) cube complex on which $G_{T}$ acts. The main element of interest is the following: given a space $X$, a \emph{wall} is a pair $\{U,V\}$ where $U\cup V=X.$

\begin{remark}\label{remark: the wallspace}
Suppose that $\Sigma_{R}-(\sigma_{r}\cup \sigma_{r'})$ contains exactly one infinite component for each $rr'\in T$. Each translate $(k\sigma_{r}\cup k\sigma_{r'})$ for $k\in K_{R}$ and $rr'\in T$ determines some finite number, $N_{r,r'}$, of \emph{walls} $$\Lambda^{i}_{k,r,r'}=\bigg\{U^{i}_{k,r,r'}\cup \pi^{-1}(kA_{r,r'}),V^{i}_{k,r,r'}\cup \pi^{-1}(kA_{r,r'})\bigg\},$$ where $\{U^{i}_{k,r,r'},V^{i}_{k,r,r'}\}$ ranges over the finite number of non-trivial partitions of connected components of $\hat{\Delta}_{T,R}-\pi^{-1}(kA_{r,r'})$. There are finitely many such choices by Lemma \ref{lem: finitely many connected components}. There is clearly a uniform upper bound $N$ to the number $N_{r,r'}$ as $rr'$ ranges over $T$.
\end{remark}

We will say that two walls $\{U,V\},\{U',V'\}$ are \emph{transverse} if each of $U\cap U',\; U\cap V',\;V\cap U, V\cap V'$ are non empty.
A wall $\{U,V\}$ \emph{separates} two points $x,y$ if $x\in U-V$ and $y\in V-U$ (or vice versa). 

\begin{remark}\label{remark: finitely many transverse}
Note that two walls $\Lambda^{i}_{k,r,r'}$ and $\Lambda^{j}_{\kappa,\rho,\rho'}$ are transverse only if $kA_{r,r'}\cap \kappa A_{\rho,\rho'}$ is non empty. Since $A_{r,r'}$ has diameter bounded above by some uniform constant $D$, there is a finite upper bound to the size of any collection of pairwise intersecting $\kappa_{i}A_{r,r'}$. Since each $kA_{r,r'}$ gives a bounded number of walls of the form $\Lambda^{i}_{k,r,r'}$, there is a finite upper bound $M$ to the number of pairwise transverse walls. 
\end{remark}
\begin{definition}
Suppose that $\Sigma_{R}-(\sigma_{r}\cup\sigma_{r'})$ contains exactly one infinite component for each $rr'\in T.$ Let $\Lambda^{i}_{k,r,r'}$ be defined as in Remark \ref{remark: the wallspace}. We define the collection of walls of $\hat{\Delta}_{T,R}$, $$\mathcal{W}=\{\Lambda^{i}_{k,r,r'}\;:\;k\in K_{R},\; r,r'\in R,\;rr'\in T\}.$$ 
\end{definition}Note that $\hat{H}_{T,R}$ acts on a wall $\Lambda^{i}_{g,r,r'}$ in a clear manner, and $\hat{H}_{T,R}\mathcal{W}=\mathcal{W}$.

\begin{remark}\label{remark: the dual cube complex}
The graph $\hat{\Delta}_{T,R}$ is locally finite, and so Remark \ref{remark: finitely many transverse} on the size of any pariwise transverse collection of $\Lambda^{i}_{k,r,r'}$ implies that the pair $(\hat{\Delta}_{T,R},\mathcal{W})$ is a \emph{wallspace} in the language of Hruska--Wise \cite[$\S 2.2$]{Hruska-Wise}. There exists a canonical CAT(0) cube complex, $\mathcal{C}=\mathcal{C}(\hat{\Delta}_{T,R},\mathcal{W})$, on which $\hat{H}_{T,R}$ acts by isometries, built from $(\hat{\Delta}_{T,R},\mathcal{W})$: see \cite[§3]{Hruska-Wise} for an overview of its construction. Since any collection of pairwise transverse $\Lambda$ in $\mathcal{W}$ has cardinality at most $M$ by Remark \ref{remark: finitely many transverse}, the cube complex $\mathcal{C}$ is finite dimensional \cite[Corollary 3.13]{Hruska-Wise}. 
\end{remark}

\begin{definition}
For two points $\hat{x},\hat{y}$ in $\hat{\Delta}_{T,R}$, let $\#(\hat{x},\hat{y})$ be the (finite) number of $\Lambda\in\mathcal{W}$ such that $\hat{x}$ and $\hat{y}$ are separated by $\Lambda$. 
\end{definition}
\begin{remark}\label{remark: points with distance equal to separation}
    Given $\hat{x},\hat{y}\in\hat{\Delta}_{T,R}$ and $h\in \hat{H}_{T,R}$ with $h\hat{x}=\hat{y}$, there exist points $c_{\hat{x}},c_{\hat{y}}\in \mathcal{C}$ with $d_{\mathcal{C}}(c_{\hat{x}},c_{\hat{y}})=\#(\hat{x},\hat{y})$ and $hc_{\hat{x}}=c_{\hat{y}}$ \cite[p. 480]{Hruska-Wise}. 
\end{remark}

We may now prove Theorem \ref{thmalph: well separated implies no (T)}.
\begin{reptheorem}{thmalph: well separated implies no (T)}
Let $G$ be an aspherically relator separated group. Then $G$ acts with unbounded orbits on a finite dimensional CAT(0) cube complex, and hence does not have Property (T).
\end{reptheorem}
\begin{proof}[Proof of Theorem \ref{thmalph: well separated implies no (T)}]

If $G=\langle S\vert T\rangle$ is aspherically relator separated, we may assume that $\vert S\cup S^{-1}\vert \geq 4$, as otherwise $G$ is cyclic. Let $G_{T},$ $\Gamma_{T},$ $K_{r},$ $\Sigma_{R},$ $\hat{H}_{T,R}$, $\hat{\Delta}_{T,R},$ $\Sigma_{R},$ be defined by our choice of aspherically relator separated group $G=\langle S\vert T\rangle$ and our fixed partition of each element $t_{i}=r_{i}r_{i}'\in T$.

For $r\in R$, let $\sigma_{r}$ be defined as in Definition \ref{def: the loop sigma r}, i.e. the loop in $\Sigma_{R}$ based at the identity and read by $r$. First suppose that $\Sigma_{R}-(\sigma_{r}\cup \sigma_{r'})$ contains exactly one infinite component for each $rr'\in T$. Let $A_{r,r'},\;\Lambda^{i}_{g,r,r'}$, $\mathcal{W}$ be as described in Lemma \ref{lem: preimages are separating}, Lemma \ref{lem: finitely many connected components}, and Remark \ref{remark: the wallspace}. Consider the CAT(0) cube complex $\mathcal{C}(\hat{\Delta}_{R},\mathcal{W}).$

We know that $\Sigma_{R}-A_{r,r'}$ is connected for each $rr'\in T$. As discussed previously in Remark \ref{remark: the wallspace}, for each $k\in K_{R}$ and $rr'\in T$, the set $kA_{r,r'}$ gives us the finite collection of walls $\{\Lambda^{i}_{k,r,r'}\}$. 

Let $n\geq 1$. Since $K_{R}$ (and therefore $\Sigma_{R}$) is infinite, and $K_{R}$ has at least two generators, we may choose a simple path $\sigma$ with endpoints $x,y\in V(\Sigma_{R})$ such that: after removing loops from $\sigma$ the resulting path is simple; there exist translates $k_{i_{1}}A_{r_{i_{1}},r'_{i_{1}}},\hdots, k_{i_{n}}A_{r_{i_{n}},r_{i_{n}}'}$ pairwise disjoint such that $\sigma$ contains exactly one loop from each of  $\{k_{i_{1}}\sigma_{r_{i_{1}}},k_{i_{1}}\sigma_{r'_{i_{1}}}\},\hdots, \{k_{i_{n}}\sigma_{r_{i_{n}}},k_{i_{n}}\sigma_{r'_{i_{n}}}\}$; and no other loops appear in $\sigma.$ Then $\sigma$ lifts to a path $\hat{\sigma}$ in $\hat{\Delta}_{T,R}$ with endpoints $\hat{x},\hat{y}$ in $V(\hat{\Delta}_{T,R})$. By Lemma \ref{lem: preimages are separating}, $\hat{x}$ and $\hat{y}$ are separated by the walls $\Lambda^{j_{1}}_{k_{i_{1}},r_{i_{1}},r'_{i_{1}}},\hdots , \Lambda^{j_{n}}_{k_{i_{n}},r_{i_{n}},r'_{i_{n}}}$ for appropriate choices of $j_{1},\hdots, j_{n}$, so that $\#(\hat{x},\hat{y})\geq n$. As $\hat{\Delta}_{T,R}$ is the Cayley graph of $\hat{H}_{T,R}$, there exists $h\in \hat{H}_{T,R}$ with $h\hat{x}=\hat{y}.$ Therefore, by Remark \ref{remark: points with distance equal to separation}, for all $n\geq 1$ there exist points $c_{\hat{x}},c_{\hat{y}}\in \mathcal{C}$ and $h\in \hat{H}_{T,R}$ with $d_{\mathcal{C}}(c_{\hat{x}},c_{\hat{y}})\geq n$ and $hc_{\hat{x}}=c_{\hat{y}}$, and so $\hat{H}_{T,R}$ acts on $\mathcal{C}$ with unbounded orbits. 
    
If the unbounded action of $\hat{H}_{T,R}$ on $\mathcal{C}$ is given by the homomorphism $\alpha:\hat{H}_{T,R}\rightarrow Isom (\mathcal{C})$, then $G=G_{T}$ also acts with unbounded orbits on $\mathcal{C}$, via the action $G_{T}\twoheadrightarrow \hat{H}_{T,R}\xrightarrow{\alpha} Isom (\mathcal{C})$.

If $\Sigma_{R}-(\sigma_{r}\cup \sigma_{r'})$ contains at least two infinite components for some $rr'\in T$, then we may repeat the above argument using $\Sigma_{R}-(\sigma_{r}\cup \sigma_{r'})$ in place of $\hat{\Delta}_{T,R}-\pi^{-1}(\sigma_{r}\cup\sigma_{r'})$ to build a wallspace $(\Sigma_{R},\mathcal{W}')$. In particular, for each $k\in K_{R}$, we take the partitions of the (necessarily finite number of) components of $\Sigma_{R}-k(\sigma_{r}\cup\sigma_{r'})$ as our walls, which gives us the wallspace $(\Sigma_{R},\mathcal{W'})$. We then similarly deduce that the group $K_{R}$ acts with unbounded orbits on the CAT(0) cube complex $\mathcal{C}(\Sigma_{R},\mathcal{W}').$
\end{proof}

\section{Random groups are aspherically relator separated}\label{sec: Random groups are aspherically relator separated}

We now turn to proving Lemma \ref{lem: random groups are well separated}. We first introduce some auxiliary models of random groups to model the behaviour of $\mathcal{H}(G)$ for $G\sim G(n,l,d)$. For $t, t'^{-1}\in S\sqcup S^{-1}$, let $W(t,t')$ be the set of words in $F(S)$ starting with $t$ and ending with $t'$.

Essentially we have the following issue. Given $G=\langle S\vert T\rangle \sim G(n,2l,d)$, we would like to say that $\mathcal{H}(G)=\langle S\vert R\rangle\sim G(n,l,2d)$. Recall that given $T=\{t_{1},\hdots ,t_{N}\}$, we will fix for each $i$ the partition $t_{i}=r_{i}r_{i}'$ with $\vert r_{i}\vert=\vert r_{i}'\vert$ if $\vert t_{i}\vert$ is even, and $\vert r_{i}\vert=\vert r_{i}\vert-1$ otherwise. In particular, $\mathcal{H}(G)=\langle S\vert R\rangle$ for $R=\{r_{1},r_{1}',\hdots ,r_{N},r_{N}'\}$ for $r_{i},r_{i}'$ as described.

However, suppose that there exists some word $r\in R\cap W(t_{1},t_{1}')$. Then there exists a reduced word $r'\in R$ such that the word $t=rr'$ is cyclically reduced without cancellation and lies in $T$ (or possibly $t=r'r$). Therefore $r'\in W(t_{2},t_{2}')$ for some $t_{2}^{-1}\neq t_{1}',\;t_{1}^{-1}\neq t_{2}'$. In particular, the relator set for $\mathcal{H}(G)$ is not chosen uniformly at random. We are therefore required to take a slight diversion for technical purposes. In essence, we show that we can extend the relator set $R$ slightly to get a `random enough' group presentation. We define:\begin{itemize}
   \item $\Theta(n,l,d)$ by choosing uniformly at random a set $R_{l}$ of $2(2n-1)^{ld}$ reduced words of length $l$, and considering $H=F_{n}\slash \langle \langle R_{l}\rangle \rangle,$

\item $\Omega(n,l,d)$ by choosing uniformly at random a set $R_{l}$ of $(2n-1)^{ld}$ reduced words of length $l$ and a set $R_{l+1}$ of $(2n-1)^{ld}$ reduced words of length $l+1$, and considering $H=F_{n}\slash \langle \langle R_{l}\cup R_{l+1}\rangle \rangle,$

\item $\overline{\Theta}(n,l,d)$ by choosing $\langle S\vert R_{l}\rangle\sim \Theta (n,l,d)$ conditioned on, for all $t, t^{-1},$  $\vert R_{l}\cap W(t,t')\vert=2(2n-1)^{ld}(2n)^{-2}.$

    \item $\overline{\Omega}(n,l,d)$ by choosing $\langle S\vert R_{l}\sqcup R_{l+1}\rangle\sim \Omega(n,l,d)$ conditioned on, for all $t, t^{-1},$ $\vert R_{l}\cap W(t,t')\vert=\vert R_{l+1}\cap W(t,t')\vert=(2n-1)^{ld}(2n)^{-2}.$
\end{itemize}
Let $f,g:\mathbb{N}\rightarrow \mathbb{R}_{+}$: we write  $f=o_{m}(g)$ if $f(m)\slash g(m)\rightarrow 0$ as $m\rightarrow\infty$, and $f=O_{m}(g)$ if there exist constants $N\geq 0$ and $M\geq 1$ such that $f(m)\leq N g(m)$ for all $m\geq M$. We now discuss spherical diagrams. 
\begin{definition}
Let $G=\langle S\vert  R\rangle$ be a group. A \emph{spherical diagram} for $G$ is a finite cellular decomposition of the sphere, $D$, such that:
\begin{enumerate}[label=$\roman*)$]
    \item Each (oriented) $1$-cell of $D$ is labelled by some $s\in S$ such that if $e$ is labelled by $s$ then $e^{-1}$ is labelled by $s^{-1}$.
    \item Each $2$-cell $C$ has a marked start point and orientation. Reading the boundary of $C$ from the marked point and concatenating edge labels, the word obtained is reduced without cancellation and is equal to some $r\in R$. 
\end{enumerate}
 A spherical diagram $D$ is \emph{unreduced} if $D$ contains $2$-cells $C_{1},C_{2}$, such that: $C_{1}$ and $C_{2}$ bear the same relator with opposite orientations; and $C_{1}$ and $C_{2}$ both contain an edge $e\in D$ that represents the same letter in the relator (with respect to marked start points and orientations). 
\end{definition}

\begin{definition}
Following Gersten \cite{gersten1987reducible}, a presentation $G=\langle S\vert R\rangle$ is \emph{diagramatically reducible} if every spherical diagram for $G$ is either unreduced or consists of a single point.
\end{definition}

Note that being diagramatically reducible is preserved by removing relators from the set $R$. Importantly, if $G$ is diagramatically reducible, then it is aspherical \cite[Remark 3.2]{gersten1987reducible}. We have the following theorem due to Ollivier \cite{olliviersharp}. This was originally proved only for the $G(n,l,d)$ model, but the result follows similarly for $\Theta (n,l,d)$ and $\Omega(n,l,d)$ (see \cite[$\S I.2.c$]{Ollivier_Jan_Invitation} or \cite[Remark 1.4]{Ollivier-Wise}).

\begin{theorem}\label{thm: vKd reduced for random groups}\cite[pp. 614--615]{olliviersharp}
Fix $d<1\slash 2$ and let $G\sim G(n,l,d)$, $G\sim \Theta(n,l,d)$, or $G\sim \Omega(n,l,d)$. Then w.o.p. $G$ is infinite and diagramatically reducible.
\end{theorem}
Using this, we may prove the following two lemmas. We will use the Chebyshev inequality \cite{Chebyshev}: if $X$ is a random variable with finite expected value $\mu$ and finite variance $\sigma^{2}$, then for any $\kappa>0$, $$P(\vert X-\mu\vert\geq \kappa\sigma)\leq \kappa^{-2}.$$
\begin{lemma}\label{lem: asphericity of model condition}
    Fix $d<1\slash 2$ and let $G=\langle S\vert R\rangle \sim \Theta(n,l,d)$, or $G=\langle S\vert R\rangle \sim \Omega(n,l,d)$, conditioned on 
    $$\bigg\vert\vert R\cap W(t,t')\vert-2(2n-1)^{ld}(2n)^{-2}\bigg\vert\leq (2n-1)^{3ld\slash 4}$$ for all $t,t'$. Then w.o.p. $G$ is infinite and diagramatically reducible.
\end{lemma}
\begin{proof}
 We prove this for $\Theta(n,l,d)$ only, since the proof is similar for $\Omega(n,l,d)$. First, let $R_{l}$ be a uniformly random choice of set of $2(2n-1)^{ld}$ reduced words of length $l$ in $F_{n}$, so that $\langle S\vert R_{l}\rangle \sim G(n,l,d)$. Fix $t,t^{-1}$.  Define the random variable $X_{t,t'}=\vert R_{l}\cap W(t,t')\vert.$ Then $X_{t,t'}$ has expectation $\mu=2(2n-1)^{ld}(2n)^{-2}$ and variance $\sigma^{2}=O_{l}(\mu).$ Therefore, taking $\kappa=\kappa(l)=\mu^{1\slash 100}$, by applying the Chebyshev inequality, we see that  $P(\vert X_{t,t'}-\mu\vert\geq \kappa\sigma)\leq \kappa^{-2}$. In particular $$\vert X_{t,t'}-\mu\vert \geq k\sigma=O_{l}(\mu^{51\slash 100})$$ with probability less than or equal to $\mu^{-2\slash 100}.$ Therefore
$$  \bigg\vert X_{t,t'}-2(2n-1)^{ld}(2n)^{-2}\bigg\vert\leq (2n-1)^{3ld\slash 4}$$ with probability tending to $1$ as $l$ tends to infinity. Taking the finite intersection of events across all pairs $t,t'$, it follows that for all $t,t'$ w.o.p. $$\bigg\vert\vert R_{l}\cap W(t,t')\vert-2(2n-1)^{ld}(2n)^{-2}\bigg\vert\leq (2n-1)^{3ld\slash 4}.$$

In particular, we have shown that the set of group presentations obtainable by sampling from $\Theta (n,l,d)$ conditioned on $$\bigg\vert\vert R\cap W(t,t')\vert-2(2n-1)^{ld}(2n)^{-2}\bigg\vert\leq (2n-1)^{3ld\slash 4}$$ for all $t,t'$ is of density $1-o_{l}(1)$ in the set of group presentations obtainable by sampling from $\Theta (n,l,d)$. 

By Theorem \ref{thm: vKd reduced for random groups}, we have that the set of group presentations obtainable by sampling from $\Theta (n,l,d)$ conditioned on the group being infinite and diagramatically reducible is also of density $1-o_{l}(1)$ in the set of group presentations obtainable by sampling from $\Theta (n,l,d)$. 

Taking the intersection of two subsets of density $1-o_{l}(1)$ also gives a subset of density $1-o_{l}(1)$, so the conclusion of the Lemma holds. In particular, given $G=\langle S\vert R_{l}\rangle \sim \Theta (n,l,d)$, conditioned on $$\bigg\vert\vert R_{l'}\cap W(t,t')\vert-2(2n-1)^{ld}(2n)^{-2}\bigg\vert\leq (2n-1)^{3d\slash 4}$$ for all $t,t'$, we have that $G$ is w.o.p. infinite diagramatically reducible.
\end{proof}
\begin{lemma}\label{lem: equivalence of models}
    Fix $d<1\slash 2$ and let $H=\langle S\vert R\rangle \sim \overline{\Theta}(n,l,d)$ or $H\sim \overline{\Omega}(n,l,d)$. Then w.o.p. $H$ is infinite and diagramatically reducible.
\end{lemma}
\begin{proof}
 We prove this for $\overline{\Theta} (n,l,d)$ only; the proof is similar for $\overline{\Omega}(n,l,d)$.

Let $R_{l}$ be a randomly chosen set of reduced words of length $l$ in $F_{n}$, conditioned on $\vert R_{l}\cap W(t,t')\vert=(2n-1)^{ld}(2n)^{-2}$ for all $t,t'$. Let $d<d'<1\slash 2$, and let $R_{l}\subseteq R_{l}'$ be a randomly chosen set of reduced words of length $l$ in $F_{n}$, conditioned on $$\bigg\vert\vert R_{l}'\cap W(t,t')\vert-2(2n-1)^{ld'}(2n)^{-2}\bigg\vert\leq (2n-1)^{3d'\slash 4}$$ for all $t,t'$. Note that the distribution of $R_{l}'$ is uniform, so that $\langle S\vert R_{l}'\rangle\sim \Theta (n,l,d')$, conditioned on $$\bigg\vert\vert R_{l}'\cap W(t,t')\vert-2(2n-1)^{ld'}(2n)^{-2}\bigg\vert\leq (2n-1)^{3d'\slash 4}$$
for all $t,t'$. By Lemma \ref{lem: asphericity of model condition}, the group $\langle S\vert R_{l}'\rangle$ is w.o.p. infinite and diagramatically reducible . Being infinite and diagramatically reducible is preserved by removing relators, and so $H=\langle S\vert R_{l}\rangle$ is w.o.p. infinite and diagramatically reducible.
\end{proof}

We now note a result on small cancellation in random groups: we refer the reader to \cite[p. 240]{LyndonCGT} for the definition of the $C'(\lambda)$ condition.
\begin{lemma}\label{lem: C'(1/2) condition for density model}\cite[Corollary 3]{ollivier2007some}
    Let $d<1\slash 2$ and let $G\sim G(n,l,d)$. Then w.o.p. $G$ is $C'(2d).$
\end{lemma}
We now analyse the group $\mathcal{H}(G)$.
\begin{lemma}\label{lem: halved group is also random}
    Let $d<1\slash 4$ and $\delta< 1\slash 4-d.$ Let $G\sim G(n,l,d)$, and $\mathcal{H}(G)=\langle S\vert R\rangle$. Then w.o.p. there exists a set $R\subseteq R'$ such that $\langle S\vert R'\rangle \sim \overline{\Theta}(n,l\slash 2, 2d+\delta)$ if $l$ is even, and $\langle S\vert R'\rangle\sim \overline{\Omega}(n,(l-1)\slash 2, 2d+\delta )$ if $l$ is odd.
\end{lemma}
\begin{proof}
    We prove the first case only. Let $G=\langle S\vert  T\rangle$ and $\mathcal{H}(G)=\langle S\vert  R\rangle$. By Lemma \ref{lem: C'(1/2) condition for density model}, w.o.p. $R$ consists of exactly $2(2n-1)^{ld}=2(2n-1)^{2d(l\slash 2)}$ reduced words of length $l\slash 2$. By the Chebyshev inequality, w.o.p. for each $t, t'^{-1}$, $$\bigg\vert\vert R\cap W(t,t')\vert-2(2n-1)^{ld}(2n)^{-2}\bigg\vert\leq (2n-1)^{3ld\slash 4}$$ so that w.o.p. for each $t, t'^{-1}$, $$\vert R\cap W(t,t')\vert\leq (2n-1)^{l(2d+\delta)\slash 2}(2n)^{-2}.$$
    
    We may therefore choose uniformly at random a set $R\subseteq R'$ with $\vert R'\cap W(t,t')\vert=2(2n-1)^{l(2d+\delta)\slash 2}(2n)^{-2}$ for all $t,t'$, so that $\langle S\vert R'\rangle\sim \overline{\Theta}(n,l\slash 2, 2d+\delta).$
\end{proof}
Using these results, we may immediately deduce Lemma \ref{lem: random groups are well separated}.
\begin{lemma}\label{lem: random groups are well separated}
Let $G$ be a random group at density less than $1\slash 4.$ Then, with overwhelming probability, $G$ is aspherically relator separated. 
\end{lemma}
\begin{proof}[Proof of Lemma \ref{lem: random groups are well separated}]
Let $d<1\slash 4$, and $G\sim G(n,l,d)$. Then
 $G$ is w.o.p. infinite and diagramatically reducible by Theorem \ref{thm: vKd reduced for random groups}, and therefore infinite and aspherical w.o.p. by \cite[Remark 3.2]{gersten1987reducible}. If $\mathcal{H}(G)=\langle S\vert R\rangle $ then w.o.p. there exists a set $R'$ such that $R\subseteq R'$ and $\langle S\vert R'\rangle $ is infinite and diagramatically reducible by Lemma \ref{lem: equivalence of models} and Lemma \ref{lem: halved group is also random}. Since both of these properties are preserved by removing relators, it follows that $\mathcal{H}(G)$ is w.o.p. infinite and diagramatically reducible, hence infinite and aspherical w.o.p. by \cite[Remark 3.2]{gersten1987reducible}.
\end{proof}
\subsection{Extension to the \texorpdfstring{$\boldsymbol{k}$}{k}-angular model}
The proof of Theorem \ref{thmalph: lack of (T) in $k$-angular model} follows an identical method to the above proof of Theorem \ref{thmalph: no property t in random groups}, obtained by simply swapping results in the density model for the analogous results in the $k$-angular model. We list the replacements below.

The following theorem, which we use in place of Theorem \ref{thm: vKd reduced for random groups}, was originally only proved for $G(n,l,d)$. However, the proof strategy of \cite[Theorem 3.11]{ashcroftrandom} can be extended to the models $\Theta (n,l,d)$ and $\Omega(n,l,d)$ with little alteration. We therefore omit the details of the proof of the following, and simply cite the result.
\begin{theorem}\cite[Theorem 3.11]{ashcroftrandom}
Let $k\geq2$, $d<1\slash 2$, and let $G\sim G(n,k,d),\;G\sim\Theta (n,k,d),\;$or $G\sim\Omega(n,k,d)$. Then $w.o.p.(n)$ $G$ is infinite and diagramatically reducible.
\end{theorem}

Next, we use the following in place of Lemma \ref{lem: C'(1/2) condition for density model}, which follows by an application of the isoperimetric inequality of random $k$-angular groups, due to Odrzyg{\'o}{\'z}d{\'z}.
\begin{lemma}\cite[Theorem 2.6]{odrzygozdz2019bent}
    Let $d<1\slash 2$ and let $G\sim G(n,k,d)$. Then w.o.p.(n) $G$ is $C'(2d).$
\end{lemma}

 Finally, in place of Lemma \ref{lem: halved group is also random}, we make the following observation.

\begin{remark}
    The main difference between $\mathcal{H}(G)=\langle S\vert R\rangle$ for $G$ a random group in the density model, as opposed to the $k$-angular model, is the following. Let $G\sim (n,k,d)$ for $k$ even. Suppose that $r=r_{1}\hdots r_{k\slash 2}\in R$. Then we must have a word $r'=r'_{1}\hdots r'_{k\slash 2}\in R$ where $r_{1}'\neq r_{k\slash 2}^{-1},\; r'_{k\slash 2}\neq r_{1}^{-1}$. In the case of a fixed number of generators, i.e. where we view $G$ as a random group in the density model (so that $k\rightarrow \infty$) there are approximately $(2n-1)^{k\slash 2}$ choices of such a word $r'$. However, there are $2n(2n-1)^{kd-1}$ reduced words of length $k$ that could be chosen from to obtain a random group in $\Theta(n,k\slash 2,2d)$. In particular, the relator set for $\mathcal{H}(G)$ is not chosen independently.

   Now suppose that we are in the case that $n\rightarrow\infty$, so that we view $G$ as a random group in the $k$-angular model. There are again approximately $(2n-1)^{k\slash 2}$ choice of $r'\in R$ with $rr'\in T$. Since $(2n-1)^{k\slash 2}\slash 2n(2n-1)^{k\slash 2-1}\rightarrow 1$, the relator set of $\mathcal{H}(G)$ is chosen asymptotically independently, so that we can view $\mathcal{H}(G)$ as having the (asymptotic) distribution of $\Theta (n,k\slash 2,2d).$
    
    If $k$ is odd, then we may similarly deduce that $\mathcal{H}(G)$ has the asymptotic distribution of $\Omega(n,(k-1)\slash 2, 2kd\slash (k-1)).$
    \end{remark}
The proof of Theorem \ref{thmalph: lack of (T) in $k$-angular model} now follows similarly to our proof of Theorem \ref{thmalph: no property t in random groups}.
	\bibliographystyle{alpha}
	\bibliography{bib}

\begin{thebibliography}{dLdlS21}

\bibitem[Ago13]{Agol13}
Ian Agol.
\newblock The virtual {H}aken conjecture.
\newblock {\em Doc. Math.}, 18:1045--1087, 2013.
\newblock With an appendix by Agol, Daniel Groves, and Jason Manning.

\bibitem[ARD20]{ashcroftrandom}
Calum~J. Ashcroft and Colva~M. Roney-Dougal.
\newblock On random presentations with fixed relator length.
\newblock {\em Comm. Algebra}, 48(5):1904--1918, 2020.

\bibitem[Ash21]{ashcroft2021property}
Calum~J Ashcroft.
\newblock Property {(T)} in density-type models of random groups.
\newblock {\em arXiv preprint arXiv:2104.14986}, 2021.

\bibitem[Ash22]{ashcroft2022propertyTrandomquotients}
Calum~J Ashcroft.
\newblock Property ({T}) in random quotients of hyperbolic groups at densities
  above 1/3.
\newblock {\em arXiv preprint arXiv:2202.12318}, 2022.

\bibitem[CCH81]{Chiswell_asphericity}
Ian~M. Chiswell, Donald~J. Collins, and Johannes Huebschmann.
\newblock Aspherical group presentations.
\newblock {\em Mathematische Zeitschrift}, 178(1):1--36, 1981.

\bibitem[dLdlS21]{de2021banach}
Tim de~Laat and Mikael de~la Salle.
\newblock Banach space actions and {$L^{2}$}-spectral gap.
\newblock {\em Analysis \& PDE}, 14(1):45--76, 2021.

\bibitem[DM19]{DRUTU_Mackay}
Cornelia Druţu and John~M. Mackay.
\newblock Random groups, random graphs and eigenvalues of p-laplacians.
\newblock {\em Advances in Mathematics}, 341:188--254, 2019.

\bibitem[Duo17]{duong}
Yen Duong.
\newblock {\em On {R}andom {G}roups: {T}he {S}quare {M}odel at {D}ensity d <
  1/3 and as {Q}uotients of {F}ree {N}ilpotent {G}roups}.
\newblock ProQuest LLC, Ann Arbor, MI, 2017.
\newblock Thesis (Ph.D.)--University of Illinois at Chicago.

\bibitem[Ger87]{gersten1987reducible}
Steve~M Gersten.
\newblock Reducible diagrams and equations over groups.
\newblock In {\em S. M. Gersten, editor, Essays in Group Theory}, volume~8 of
  {\em Mathematical Sciences Research Institute Publications}, pages 15--73.
  Springer, NY, 1987.

\bibitem[Gro87]{Gromov_hyperbolic}
Mikhael Gromov.
\newblock Word hyperbolic groups.
\newblock In {\em S. M. Gersten, editor, Essays in Group Theory}, volume~8 of
  {\em Mathematical Sciences Research Institute Publications}, pages 75--264.
  Springer, NY, 1987.

\bibitem[Gro93]{gromovasymptotic}
M.~Gromov.
\newblock Asymptotic invariants of infinite groups.
\newblock In {\em Geometric group theory, {V}ol. 2 ({S}ussex, 1991)}, volume
  182 of {\em London Math. Soc. Lecture Note Ser.}, pages 1--295. Cambridge
  Univ. Press, Cambridge, 1993.

\bibitem[HW08]{Haglund-Wise}
F.~Haglund and D.T. Wise.
\newblock Special cube complexes.
\newblock {\em Geometric and Functional Analysis}, 17(5):1551--1620, 2008.

\bibitem[HW14]{Hruska-Wise}
G.C. Hruska and D.T. Wise.
\newblock Finiteness properties of cubulated groups.
\newblock {\em Compositio Mathematica}, 50(3):453--506, 2014.

\bibitem[KK13]{kotowski}
Marcin Kotowski and Michał Kotowski.
\newblock Random groups and property {(T)}: {\.Z}uk's theorem revisited.
\newblock {\em Journal of the London Mathematical Society}, 88(2):396--416,
  2013.

\bibitem[LS15]{LyndonCGT}
Roger~C Lyndon and Paul~E Schupp.
\newblock {\em Combinatorial Group Theory}, volume~89 of {\em Classics in
  Mathematics}.
\newblock Springer Berlin / Heidelberg, Berlin, Heidelberg, 2015.

\bibitem[Mon21]{montee2021random}
MurphyKate Montee.
\newblock Random groups at density $d<3/14$ act non-trivially on a {CAT}(0)
  cube complex.
\newblock {\em arXiv preprint arXiv:2106.14931}, 2021.

\bibitem[Mon22]{montee(T)}
MurphyKate Montee.
\newblock Property (t) in k-gonal random groups.
\newblock {\em Glasgow Mathematical Journal}, page 1–5, 2022.

\bibitem[MP15]{mackay_przytycki2015balanced}
John~M Mackay and Piotr Przytycki.
\newblock Balanced walls for random groups.
\newblock {\em The Michigan Mathematical Journal}, 64:397--419, 2015.

\bibitem[NR97]{Niblo-Reeves97}
G.~Niblo and L.~Reeves.
\newblock Groups acting on {CAT(0)} cube complexes.
\newblock {\em Geometry \& Topology}, 1(1):1--7, 1997.

\bibitem[Odr18]{odrzygozdzsquare}
Tomasz Odrzyg{\'o}{\'z}d{\'z}.
\newblock Cubulating random groups in the square model.
\newblock {\em Israel Journal of Mathematics}, 227(2):623--661, 2018.

\bibitem[Odr19]{odrzygozdz2019bent}
Tomasz Odrzyg{\'o}{\'z}d{\'z}.
\newblock Bent walls for random groups in the square and hexagonal model.
\newblock {\em arXiv preprint arXiv:1906.05417}, 2019.

\bibitem[Oll04]{olliviersharp}
Y.~Ollivier.
\newblock Sharp phase transition theorems for hyperbolicity of random groups.
\newblock {\em Geometric and Functional Analysis}, 14(3):595--679, 2004.

\bibitem[Oll05]{Ollivier_Jan_Invitation}
Yann Ollivier.
\newblock {\em A {J}anuary 2005 invitation to random groups}, volume~10 of {\em
  Ensaios Matem\'{a}ticos [Mathematical Surveys]}.
\newblock Sociedade Brasileira de Matem\'{a}tica, Rio de Janeiro, 2005.

\bibitem[Oll07]{ollivier2007some}
Yann Ollivier.
\newblock Some small cancellation properties of random groups.
\newblock {\em International Journal of Algebra and Computation},
  17(01):37--51, 2007.

\bibitem[Opp21]{Oppenheim21}
Izhar Oppenheim.
\newblock Banach {{\.Z}}uk's criterion for partite complexes with application
  to random groups.
\newblock {\em arXiv preprint arXiv:2112.02929}, 2021.

\bibitem[Osa18]{osajda2018group}
Damian Osajda.
\newblock Group cubization.
\newblock {\em Duke Mathematical Journal}, 167(6):1049--1055, 2018.

\bibitem[OW11]{Ollivier-Wise}
Y.~Ollivier and D.T. Wise.
\newblock Cubulating random groups at density less than 1/6.
\newblock {\em Transactions of the American Mathematical Society},
  363(9):4701--4733, 2011.

\bibitem[Sab58]{SabidussiGert1958OaCo}
Gert Sabidussi.
\newblock On a class of fixed-point-free graphs.
\newblock {\em Proceedings of the American Mathematical Society},
  9(5):800--804, 1958.

\bibitem[Sag95]{Sageev-95}
M.~Sageev.
\newblock Ends of group pairs and non-positively curved cube complexes.
\newblock {\em Proceedings of the London Mathematical Society}, 3(3):585--617,
  1995.

\bibitem[Tch67]{Chebyshev}
P.~Tchébychef.
\newblock Des valeurs moyennes ({T}raduction du russe, {N}. de {K}hanikof.
\newblock {\em Journal de Mathématiques Pures et Appliquées}, pages 177--184,
  1867.

\bibitem[Wis21]{Wise-MSQT}
Daniel~T Wise.
\newblock {\em The structure of groups with a quasiconvex hierarchy}.
\newblock Princeton University Press, 2021.

\bibitem[\.{Z}03]{Zuk}
A.~\.{Z}uk.
\newblock {P}roperty {(T)} and {K}azhdan constants for discrete groups.
\newblock {\em Geometric and Functional Analysis}, 13(3):643--670, 2003.

\end{thebibliography}
	{\sc{DPMMS, Centre for Mathematical Sciences, Wilberforce Road, Cambridge, CB3 0WB, UK}.} \textit{\textsc{\textit{E-mail address}}:}\textsc{ cja59@dpmms.cam.ac.uk}
\end{document}